\documentclass{amsart}
\usepackage{enumerate,color}

\newtheorem{theorem}{Theorem}[section]
\newtheorem{lemma}[theorem]{Lemma}

\newtheorem{cor}[theorem]{Corollary}
\newtheorem{thm}{Theorem}

\theoremstyle{definition}
\newtheorem{remark}[theorem]{Remark}

\renewcommand{\Re}{\operatorname{Re}}
\renewcommand{\Im}{\operatorname{Im}}

\newcommand{\D}{\operatorname{\{s\in\mathbb{C}:\tfrac{1}{2}<\Re(s)<1\}}}
\newcommand{\meas}[1]{\liminf_{T\to\infty}\frac{1}{T}\operatorname{meas}\left\{\tau\in [0,T]: #1\right\}}

\begin{document}
\title[Joint universality and generalized strong recurrence]{Joint universality and generalized strong recurrence with rational parameter}
\author{{\L}ukasz Pa\'nkowski}
\address{Faculty of Mathematics and Computer Science, Adam Mickiewicz University, Umultowska 87, 61-614 Pozna\'{n}, Poland, and Graduate School of Mathematics, Nagoya University, Nagoya, 464-8602, Japan}
\email{lpan@amu.edu.pl}

\thanks{The author was partially supported by (JSPS) KAKENHI grant no. 26004317 and the grant no. 2013/11/B/ST1/02799 from the National Science Centre.}



\begin{abstract}

We prove that, for every rational $d\ne 0,\pm 1$ and every compact set $K\subset\{s\in\mathbb{C}:1/2<\Re(s)<1\}$ with connected complement, any analytic non-vanishing functions $f_1,f_2$ on $K$ can be approximated, uniformly on $K$, by the shifts $\zeta(s+i\tau)$ and $\zeta(s+id\tau)$, respectively. As a consequence we deduce that
the set of $\tau$ satisfying $|\zeta(s+i\tau)-\zeta(s+id\tau)|<\varepsilon$ uniformly on $K$ has a positive lower density for every $d\ne 0$.
\end{abstract}

\maketitle

\section{Introduction}

In 1981 Bagchi \cite{B81} discovered an interesting connection between the Riemann Hypothesis and Voronin's universality theorem  (see \cite{V}) for the Riemann zeta function $\zeta(s)$. Namely, he proved that $\zeta(s)\ne 0$ for $\Re(s)>\tfrac{1}{2}$ if and only if for every compact set $K\subset\D$ with connected complement and every $\varepsilon$ we have
\begin{equation}\label{eq:SR}
\meas{\max_{s\in K}|\zeta(s+i\tau)-\zeta(s)|<\varepsilon}>0,
\end{equation}
where $\operatorname{meas}\{\cdot\}$ denotes the real Lebesgue measure.
In the language of topological dynamics (see \cite{GH}) \eqref{eq:SR} is called the \emph{strong recurrence} property for the Riemann zeta function.

Bagchi's observation was extended to the case of Dirichlet $L$-functions by himself in \cite{B82} and \cite{B87}, and to the case of general universal L-functions, for which the Generalized Riemann Hypothesis is expected, in \cite[Theorem 8.4]{S}.

Nakamura \cite{N09} suggested the following related problem: find all $d$ such that for every compact set $K\subset\D$ with connected complement and every $\varepsilon$ we have
\begin{equation}\label{eq:Self}
\meas{\max_{s\in K}|\zeta(s+i\tau)-\zeta(s+id\tau)|<\varepsilon}>0.
\end{equation}
This property can be called \emph{generalized strong recurrence} with parameter $d$. However, it should be noted that sometimes in the literature it is called also the self-approximation property with parameter $d$. Using this notion Bagchi's result states that the Riemann Hypothesis is equivalent to the generalized strong recurrence property for $\zeta(s)$ with parameter $d=0$.

Nakamura, in the same paper, gave the partial answer to this question by proving that \eqref{eq:Self} holds if $d$ is algebraic irrational. He also observed that the generalized strong recurrence property holds for almost all real parameters $d$. His result was improved by the author in \cite{P09} to all irrational parameter $d$. The positive answer for non-zero rational $d$ was claimed by Garunk\v{s}tis \cite{Ga} and Nakamura \cite{N10}. Unfortunately, their arguments have a gap, which was pointed out by Nakamura and Pa\'nkowski \cite{NP} and partially filled, in the same paper, for all non-zero rational $d=\tfrac{a}{b}$ with $\gcd(a,b)=1$ and $|a-b|\ne 1$.

The crucial step in the proof of the generalized strong recurrence property with parameter $d$ is to show that the following set
\begin{equation}\label{eq:set}
\left\{\log p: p\text{ is prime}\right\}\cup \left\{d\log p: p\text{ is prime}\right\}
\end{equation}
is linearly independent over $\mathbb{Q}$. It was proved for all algebraic irrational $d$ and for almost all $d$ by Nakamura \cite{N09}. Moreover, by using the six exponential theorem from the theory of transcendental numbers, the author noticed in \cite{P09} that for a given irrational $d$ only a finite number of primes $p$ can possibly be involved in the linear dependence of \eqref{eq:set}. It allowed to prove the following joint universality theorem, which easily implies the generalized strong recurrence property. It was showed by Nakamura for algebraic irrational $d$ and by the author for all irrational $d$.
\begin{thm}\label{th:jointA}
Let $d$ be irrational, $K\subset \D$ be a compact set with connected complement and $f,g$ be continuous non-vanishing functions on $K$, which are analytic in the interior of $K$. Then, for every $\varepsilon>0$, we have
\begin{equation}\label{eq:jointA}
\meas{\begin{array}{l}
\max_{s\in K}\left|\zeta(s+i\tau)-f(s)\right|<\varepsilon\\
\max_{s\in K}\left|\zeta(s+id\tau)-g(s)\right|<\varepsilon
\end{array}
}>0.
\end{equation}
\end{thm}

The above joint universality theorem is also related the following open problem introduced by Andreas Weiermann in 2008 during the conference ``New Directions in the Theory of Universal Zeta- and L-Functions'' in Wu\"rzburg: \\
{\it Assume $a,b$ are transcendental and algebraically independent and functions $f, g$ satisfy the assumptions of the universality theorem. Can we find one single real $\tau$ such that $f$ is approximated by $\zeta(s + ia\tau)$ and $g$ is approximated by $\zeta(s + ib\tau)$ and both approximations are uniformly as usual?}\\
Thus, Theorem A implies that the above open problem is true even for linearly independent real non-zero numbers $a$, $b$.

The case when $d$ is rational is more delicate, since one can easily observe that the set \eqref{eq:set} is linearly dependent over $\mathbb{Q}$, even if we exclude a finite number of primes. So, in order to prove \eqref{eq:Self} for rational $d=\tfrac{a}{b}$ with $|a-b|\ne 1$ and $\gcd(a,b)=1$,  Nakamura and the author \cite{NP} proved only that \eqref{eq:jointA} holds for one common function $f=g$ depending on $d$. Moreover, by the lack of linear independence over $\mathbb{Q}$ of \eqref{eq:set}, it was expected that Theorem \ref{th:jointA} with arbitrary given functions $f,g$ cannot hold for rational  $d$. However, in the present paper we introduce the approach which allow to overcome the fact that \eqref{eq:set} is not linearly independent over $\mathbb{Q}$ and we prove the following joint universality theorem. This theorem solves completely Weiermann's problem, since it is obvious that we cannot expect a positive answer if $a=b$ or $a=-b$ by the fact that $\zeta(\overline{s})=\overline{\zeta(s)}$.

\begin{theorem}\label{th:main}
Let $a,b\in\mathbb{Z}\setminus\{0\}$ with $\tfrac{a}{b}\ne \pm 1$. Assume that $K\subset\{s\in\mathbb{C}:\tfrac{1}{2}<\Re(s)<1\}$ is a compact set with connected complement and $f_a,f_b$ are non-vanishing continuous on $K$ and analytic in the interior of $K$. Then, for every $\varepsilon>0$, we have
\[
\meas{\max_{c\in\{a,b\}}\max_{s\in K}\left|\zeta(s+ic\tau)-f_c(s)\right|<\varepsilon}>0.
\]
\end{theorem}
\begin{remark}
It should be mentioned  that the above theorem can be easily generalized to the wide class of $L$-functions, for which a universality theorem is proved by Voronin's approach (for example for a wide class introduced in \cite[Chapter VII, Section 3.1]{KV} or \cite[Section 3]{KK}).

Moreover, the above theorem together with Theorem A can be treated as a new method how to approximate more than one function by certain modifications of one zeta or $L$-function. Indeed, the above results say that if we desire to approximate two analytic non-vanishing functions $f,g$ by a given $L$-function $L(s)$ it suffices  to consider the shifts $L(s+i\tau)$, $L(s+id\tau)$, where $d$ is non-zero real number $\ne \pm 1$. A different well-known method of this kind is to consider twists of $L(s)$ with pairwise non-equivalent Dirichlet characters (see \cite[Theorem 12.8]{S}). 
\end{remark}

As an immediate consequence of Theorem \ref{th:main} we obtain the following corollary.
\begin{cor}
Let $d\ne 0,\pm 1$ be a rational number. Assume that $K\subset\{s\in\mathbb{C}:\tfrac{1}{2}<\Re(s)<1\}$ is a compact set with connected complement and $f$, $g$ are non-vanishing continuous on $K$ and analytic in the interior of $K$. Then, for every $\varepsilon>0$, we have
\[
\meas{\begin{array}{l}
\max_{s\in K}\left|\zeta(s+i\tau)-f(s)\right|<\varepsilon\\
\max_{s\in K}\left|\zeta(s+id\tau)-g(s)\right|<\varepsilon
\end{array}
}>0.
\]
\end{cor}

Obviously, taking $f\equiv g$ in the above corollary proves that \eqref{eq:Self} is true for all rational $d\ne 0,\pm 1$. However, for $d=1$ the inequality \eqref{eq:Self} holds trivially, and the generalized strong recurrence property for $d=-1$ is implied by $\overline{\zeta(s+i\tau)}=\zeta(\overline{s}-i\tau)$ and the fact that, by Voronin's theorem, we have $\max_{s\in K\cup\overline{K}}|\zeta(s+i\tau) -1|<\varepsilon$. Therefore, the following result holds.
\begin{theorem}
Let $d\ne 0$ be a real number and $K\subset\{s\in\mathbb{C}:\tfrac{1}{2}<\Re(s)<1\}$ be a compact set with connected complement. Then, for every $\varepsilon>0$, we have
\[
\meas{\max_{s\in K}\left|\zeta(s+i\tau)-\zeta(s+id\tau)\right|<\varepsilon
}>0.
\]
\end{theorem}

The above theorem reduces Nakamura's question to the case when $d=0$, which, as we mentioned before, is equivalent to the Riemann Hypothesis.

\section{Denseness lemma}

The so-called denseness lemma (see Lemma \ref{lem:dense} below) plays a crucial role in the proof of our main theorem, and, essentially, contains the main idea of this paper how to overcome the lack of linear independence of \eqref{eq:set}. In order to prove it we need the following lemmas concerning analytic functions of exponential type.

\begin{lemma}[{\cite[Lemma 6]{M}}]\label{lem:M}
Let $G(z)$ be an analytic function satisfying
\[
0\not\equiv G(z)=\sum_{m=0}^\infty\frac{\alpha_m}{m!}z^m,\qquad |\alpha_m|<A^m
\]
for some positive constant $A$. Let $c_1>0$ and $N_1$ be a positive integer. Then there exist a positive $c_2$ and a positive integer $N_2>N_1$ such that for any sufficiently large $x$ the interval $[x,x+c_1x^{-N_1}]$ contains a subinterval $I$ of length $|I|\geq c_2x^{-N_2}$ such that $G(t)$ has no zeros on $I$.
\end{lemma}

\begin{lemma}\label{lem:Arg}
Let $U\subset\mathbb{C}$ be a simply connected bounded smooth Jordan domain with $\overline{U}\subset D$. Assume that for all $s\in U$ we have $1/2<\sigma_1<\Re  s< \sigma_2<1$ and $g_1$, $g_2$ are non-zero elements of the Bergman space $B^2(U) = \{f\in L^2(U):f\text{ is holomorphic on $U$}\}$. For $z\in\mathbb{C}$ we put
\[
G_j(z)=\iint_G e^{-sz}\overline{g_j(s)}d\sigma dt,\qquad j=1,2.
\]
Then for every $\eta$ with $0<\eta<\pi/2$ there exist a sequence $x_n$ tending to $\infty$ and intervals $I_n\subset[x_n,x_n+1]$ of length $|I_n|\geq Bx_n^{-N}$, ($N>0$, $B:=B(U,\eta)>0$), such that for all $t\in I$ we have
\[
|G_1(t)|\gg e^{-\sigma_2 x_n}
\]
and, moreover, the argument of $G_1(t)$ and $G_2(t)$ on $I$ varies less than $\eta$.
\end{lemma}
\begin{proof}
Firstly, let us find a sequence $x_n$. Notice that $G_1\not\equiv 0$, since otherwise, taking the $n$-th derivative of $G_1$ at point $z=0$ and using the fact that the linear space of polynomials is dense in the Bergman space (see for example \cite[Theorem 7.2.2]{QQ}), we get the contradiction with $g\ne 0$. On the other hand, $G_1(z)\ll e^{c|z|}$ for some positive constant $c$ depending on $U$, and for sufficiently small $\omega=\omega(U)>0$ and for all complex $z$ with $|\arg(-z)|\leq \omega$ we have
\[
|e^{\sigma_2 z}G_1(z)|\ll 1.
\]
Hence, by \cite[Lemma 3]{KK}, which proof based on the Phragm\'en-Lindel\"of principle, there exists a real sequence $x_n$ tending to $\infty$ such that
\[
|G_1(x_n)|\gg e^{-\sigma_2 x_n}.
\]

Let us fix $n$ and put $x=x_n$. As in the proof of \cite[Lemma 4]{KK} one can prove that for $t\in[x,x+1]$  and every $C>0$ we have
\[
G_1(t)=P(t)+O(e^{-Ct}),
\]
where $P(t)$ is a polynomial of degree $\ll x$. 

Let $x_0\in[x,x+1]$ be such that $|P(x_0)| = \max_{x\leq t\leq x+1}|P(t)|$. Then by Markoff's inequality (see eg. \cite{Sch}) we have
\[
\max_{x\leq t\leq x+1}|P'(t)|\ll x^2|P(x_0)|
\]
and hence for $t\in[x,x+1]$ satisfying $|t-x_0|\leq \frac{B_0}{x_0^2}$ with sufficiently small $B_0>0$ we have
\begin{equation}\label{eq:meanValue}
\begin{split}
|P(x_0)|-|P(t)|&\leq |P(t)-P(x_0)|\leq \max_{x\leq t\leq x+1}|P'(t)||t-x_0|\ll B_0|P(x_0)|\\
&\leq \sin\left(\frac{\eta}{2}\right)|P(x_0)|,
\end{split}
\end{equation}
so
\[
|P(t)|\geq \left(1-\sin\left(\frac{\eta}{2}\right)\right)|P(x_0)|\geq \left(1-\sin\left(\frac{\eta}{2}\right)\right)|P(x)|.
\]
Therefore, for $t\in I_0:=[x,x+1]\cap [x_0-B_0/x_0^2,x_0+B_0/x_0^2]$ it holds
\begin{align*}
\left(1-\sin\left(\frac{\eta}{2}\right)\right)|G_1(x)|&\leq \left(1-\sin\left(\frac{\eta}{2}\right)\right)|P(x)|+O\left(e^{-Cx}\right)\\
&\leq |P(t)|+O\left(e^{-Cx}\right)\leq |G_1(t)|+O\left(e^{-Cx}\right)
\end{align*}
and hence
\[
|G_1(t)|\geq \left(1-\sin\left(\frac{\eta}{2}\right)\right)|G_1(x)|+O\left(e^{-Cx}\right)\gg e^{-\sigma_2 x}.
\]
Using again \eqref{eq:meanValue} we get that
\begin{align*}
\left|\frac{G_1(t)}{G_1(x_0)}-1\right|&\leq \frac{|P(t)-P(x_0)|+O(e^{-Cx})}{|G_1(x_0)|}\leq \frac{\sin\left(\frac{\eta}{2}\right)|P(x_0)|+O(e^{-Cx})}{|G_1(x_0)|}\\
&\leq \sin\left(\frac{\eta}{2}\right)+O(e^{-(C-\sigma_2)x}).
\end{align*}
Thus $\left|\arg\frac{G_1(t)|}{G_1(x_0)}\right|\leq \eta$ on $I_0$ for sufficiently large $x$.

Now, we use Lemma \ref{lem:M} to find subinterval $I$ of $I_0$ such that $\arg G_2(t)$ varies at most $\eta$ on $I$. The fact that $g_2$ is analytic implies (see the proof of \cite[Lemma 7.1]{KV}) that
\[ 
G_2(z)=\sum_{m=0}^\infty\frac{\alpha_m}{m!}z^m,\qquad |\alpha_m|<A^m
\]
for some $A>0$. Moreover, since $g_2\ne 0$, we have $G_2\not\equiv 0$. 

Let us define $\beta_m=\frac{\alpha_m+\overline{\alpha_m}}{2}$, $\gamma_m=\frac{\alpha_m-\overline{\alpha_m}}{2i}$ and put
\[
G^{1}_2(z) = \sum_{m=0}^\infty\frac{\beta_m}{m!}z^m,\qquad G^{2}_2(z) = \sum_{m=0}^\infty\frac{\gamma_m}{m!}z^m.
\]
Then for any real $t$ we have $G^{1}_2(t) = \Re G_2(t)$ and $G^{2}_2(t) = \Im G_2(t)$. 

Now, by Lemma \ref{lem:M} to $G^{1}_2(z)$ and $G^{2}_2(z)$, we can find subinterval $I_1\subset I$ of length $|I_1|\geq B_1x^{-N_1}$ such that $\Re G_2(t)$ and $\Im G_2(t)$ have no zeros on $I_1$. Therefore, there exists $k_1\in\{0,1,2,3\}$ such that for $t\in I_1$ we have
\[
\frac{k_1}{2}\pi\leq \arg G_2(t)\leq \frac{k_1+1}{2}\pi,
\]
so the argument of $G_2(t)$ on $I_1$ varies less than $\pi/2$.

Next, repeating the above argument for $G_3(z) = \exp(-\tfrac{k_1\pi i}{2} - \tfrac{\pi i}{4})G_2(z)$ instead of $G_2$, gives that there is a subinterval $I_2\subset I_1$ of lenght $|I_2|\geq B_2x^{-M_2}$ such that for $t\in I_2$ we have
\[
\frac{k_1}{2}\pi+\frac{k_2}{4}\pi\leq \arg G_2(t)\leq \frac{k_1}{2}\pi+\frac{k_2+1}{4}\pi
\]
for suitable $k_2\in\{0,1\}$, and hence the argument of $G_2(t)$ on $I_2$ varies less than $\pi/4$.

Thus, applying this reasoning sufficiently many times we can prove that there is an interval $I\subset I_0$ of length $|I|\geq Bx^{-N}$ such that the argument of $G_2(t)$ varies less than $\eta$, and the proof is complete.
\end{proof}

In the sake of simplicity, for a finite set $M$ of prime numbers and real numbers $\theta_p$, $p\in M$, define
\[
\zeta_M(s,(\theta_p) =\prod_{p\in M}\left(1-\frac{e(\theta_p)}{p^s}\right)^{-1},
\]
where, as usual, $e(t) = \exp(2 \pi i t)$. Moreover, let us call an open bounded subset $U$ of $\mathbb{C}$ admissible when for every sufficiently small positive $\varepsilon$ the set $U_\varepsilon:=\{s\in \mathbb{C}: \exists_{s_0\in U} |s-s_0|<\varepsilon\}$ has connected complement.

Now we are ready to formulate and prove the denseness lemma, which proof based on the following generalization of the classical Riemann rearrangement theorem.

\begin{lemma}[{\cite{Pe}}]\label{lem:Pech}
Let $H$ be a real Hilbert space and let $u_n\in H$ be such that $\sum_{n=1}^\infty\Vert u_n\Vert^2 < \infty$. Assume that for every $e\in H$ with $\Vert e\Vert = 1$ the series $\sum_{n=1}^\infty \langle u_n|e\rangle$ are conditionally converges after suitable permutation of terms. Then, for every $v\in H$ there exists a permutation $(n_k)$ such that $\sum_{k=1}^\infty u_{n_k} = v$.
\end{lemma}

\begin{lemma}\label{lem:dense}
Let $U$ be an admissible set satisfying $\overline{U}\subset \{s\in\mathbb{C}:1/2<\Re(s)<1\}$ and $a,b\in\mathbb{Z}\setminus\{0\}$ with $a\ne\pm b$. Then there exists a sequence $\theta_p$ of real numbers indexed by primes such that for any analytic non-vanishing functions $f_a,f_b$ on $\overline{U}$, $\varepsilon>0$ and $y>0$ there exists a finite set of primes $M$ containing all primes $p\leq y$ such that
\[
\max_{c\in\{a,b\}}\max_{s\in \overline{U}}\left|\zeta_M(s,(c\theta_p)) - f_c(s)\right|<\varepsilon.
\]
\end{lemma}
\begin{proof}
Let $U_1$ be a simply connected smooth Jordan domain such that $U_1$ is admissible, $f_a$, $f_b$ are analytic non-vanishing on $\overline{U_1}$ and $\overline{U}\subset U_1\subset\overline{U_1}\subset \{s\in\mathbb{C}:\sigma_1<\Re(s)<\sigma_2\}$ for suitable $\sigma_1,\sigma_2$ with $1/2<\sigma_1<\sigma_2<1$. For certain analytic $g_a$, $g_b$ we have $f_c =\exp g_c$ for $c\in\{a,b\}$.

Without loss of generality we can assume that $|a|<|b|$. Define
\[
u_p(s) = \left(-\log\left(1-\frac{e(a\theta_p)}{p^s}\right),-\log\left(1-\frac{e(b\theta_p)}{p^s}\right)\right)
\]
and 
\[
u^*_p(s) = \left(\frac{e(a\theta_p)}{p^s},\frac{e(b\theta_p)}{p^s}\right),
\]
where $\theta_{p_n} = \frac{n}{l}$, $p_n$ denotes the $n$-th prime number and $l$ is a positive integer depending on $a,b$, which we choose later.

We shall use Lemma \ref{lem:Pech} for the real Hilbert space $B^2(U_1)\times B^2(U_1)$ with the inner product given by
\[
\langle \phi,\psi\rangle = \sum_{j=1}^2\Re\iint_{U_1}\phi_j(s)\overline{\psi_j(s)}d\sigma dt
\]
for $\phi = (\phi_1,\phi_2)$, $\psi=(\psi_1,\psi_2)$.

We are going to prove that for any $\phi=(\phi_1,\phi_2)\in B^2(U_1)\times B^2(U_1)$ with $||\phi||=1$ there exists a permutation of the series $\sum_p u_p(s)$, which converges to $\phi$. Then, using the fact that $|f(z)|\leq\tfrac{||f||}{\sqrt{\pi}\operatorname{dist}(z,\partial{U_1})}$ for any analytic function $f$ and $z$ lying in the interior of $U_1$ (see \cite[Chapter III, Lemma 1.1]{G}), we get that approximation in respect to $L^2$ norm on $U_1$ implies uniform approximation on $U$, provided $\overline{U}\subset U_1$, which completes the proof.

Obviously, since $\Re(s)>\sigma_1>1/2$ for every $s\in\overline{U_1}$, we have $\sum_p||u_p(s)||^2<\infty$. Hence it suffices to prove that there are two permutations of  the series $\sum_p \langle u_p(s),\phi\rangle$ tending to $+\infty$ and $-\infty$, respectively. In fact, we show only the existence of a permutation of the series, which diverges to $+\infty$, since the case $-\infty$ is similar and can be left to the reader.
Moreover, let us observe that it is sufficient to prove it for $u^*_p(s)$ instead of $u_p(s)$, since $\sum_p (u_p(s)-u^*_p(s))$ converges absolutely for $\Re(s)>1/2$.

The case when $\phi_1=0$ or $\phi_2=0$ can be treated in the same way, so without loss of generality assume that $\phi_2=0$. Then we have to show that some permutation of the series
\[
\sum_p\langle  u^*_p,\phi\rangle = \sum_p \Re e(a\theta_p)\iint_{U_1}\frac{1}{p^s}\overline{\phi_1(s)} d\sigma dt
\]
diverges to $+\infty$.

By Lemma \ref{lem:Arg} we can show that there are infinitely many intervals $I=[x,x+Bx^{-N}]$ with $B,N>0$ such that $G_1(\log p) = \iint_{U_1}p^{-s}\overline{\phi_1(s)} d\sigma dt\gg \exp(-\sigma_2 x)$ and $|\arg G_1(\log p)-\omega_1|\leq \tfrac{\pi}{4}$ for suitable $\omega_1\in [-\pi,\pi]$, provided $\log p\in I$. Hence for sufficiently large $l>0$ there is an integer $k$ with $0\leq k<l$ such that $\arg e(ak/l)G_1(\log p)\in [-\pi/3,\pi/3]$, which implies that, if $\log p_n\in I$ and $n\equiv k\bmod l$, then $\Re e(e\theta_p)G_1(\log p)\geq c_1\exp(-\sigma_2 x)$ for some $c_1>0$. This together with $\sigma_2<1$ and the fact that the number of primes $p_n$ satisfying $\log p_n\in I$ and $n\equiv k\bmod l$ is $\gg e^x/x^{N+2}$ shows that there is a permutation $(n_k)$ such that $\sum_{k}\langle u^*_{p_{n_k}},\phi\rangle =+\infty$.

Next let us consider the case when $\phi_1\ne 0$ and $\phi_2\ne 0$. We have to show that there is a permutation $(n_k)$ such that
\[
\sum_{k}\langle  u^*_{p_{n_k}},\phi\rangle =\sum_{k}\Re e(a\theta_{p_{n_k}})G_1(\log p_{n_k})+\Re e(b\theta_{p_{n_k}})G_2(\log p_{n_k}) = +\infty,
\]
where
\[
G_j(z) = \iint_{U_1} e^{-sz}\overline{\phi_j(s)}d\sigma dt,\qquad j=1,2.
\]
Again, by Lemma \ref{lem:Arg}, we see that there exist infinitely many intervals $I=[x,x+Bx^{-N}]$ with $B,N>0$ such that
\[
G_1(t)\gg e^{-\sigma_2 x},\qquad t\in I
\]
and for every $\eta>0$ there are $\omega_1,\omega_2\in [-\pi,\pi]$ such that
\[
|\arg G_j(t) - \omega_j|\leq \eta,\qquad t\in I,\ j=1,2.
\]
We shall show that for sufficiently large $l$ there is $k$ with $0\leq k<l$ such that for $t\in I$ we have
\begin{equation}\label{eq:arg}
\arg e(ak/l)G_1(t)\in \left[-\frac{\pi}{2}+\eta,\frac{\pi}{2}-\eta\right]\qquad\text{and}\qquad \arg e(bk/l)G_2(t)\in \left[-\frac{\pi}{2},\frac{\pi}{2}\right].
\end{equation}
Then for $p_n$ satisfying $\log p_n\in I$ and $n\equiv k\bmod l$ we have
\[
\Re e(a\theta_{p_n})G_1(\log p_n)\geq  c_1e^{-\sigma_2 x}\quad(c_1>0)\qquad\text{and}\qquad \Re e(b\theta_{p_n})G_2(\log p_n)\geq 0.
\]
Hence
\[
\sum_{\substack{\log p_n\in I\\n\equiv k\bmod l}}\langle  u^*_{p_{n}},\phi\rangle\geq c'_1\frac{e^{(1-\sigma_2)x}}{x^{N+2}}
\]
for some positive constant $c'_1$.

In order to prove the existence of such $k$, notice that for every $\theta$ with $|2\pi a\theta + \omega_1|\leq \pi/2-2\eta$ we have
\[
\arg e(a\theta)G_1(t)\in \left[-\frac{\pi}{2}+\eta,\frac{\pi}{2}-\eta\right],\qquad t\in I,
\]
and 
\[
\left|2\pi b\theta+\frac{b}{a}\omega_1\right| = \frac{|b|}{|a|}|2\pi a\theta + \omega_1|\leq \frac{|b|}{|a|}\frac{\pi}{2} - \frac{|b|}{|a|}2\eta.
\]
From the assumption $|b|>|a|$ it is easy to observe that for sufficiently small $\eta:=\eta(a,b)>0$ the right hand side is at least $\tfrac{\pi}{2} + \tfrac{\pi}{4|a|}$. Hence the set
\[
\mathcal{A} = \left\{2\pi b\theta+\omega_2:|2\pi a\theta + \omega_1|\leq \pi/2-2\eta\right\}
\]
covers all values in the interval 
\[
\left[-\frac{|b|}{|a|}\frac{\pi}{2} + \frac{|b|}{|a|}2\eta+\omega_2-\frac{b}{a}\omega_1,\frac{|b|}{|a|}\frac{\pi}{2} - \frac{|b|}{|a|}2\eta+\omega_2-\frac{b}{a}\omega_1\right]
\] of length $\geq \pi + \tfrac{\pi}{2|a|}$. Then, for sufficiently small $\eta$, the set $\mathcal{A}$ contains an interval of size $\geq\tfrac{\pi}{4|a|}$ for which $\arg e(b\theta)G_2(t)\in \left[-\frac{\pi}{2},\frac{\pi}{2}\right]$, $t\in I$. Therefore, the set of $\theta$ satisfying
\[
\arg e(a\theta)G_1(t)\in \left[-\frac{\pi}{2}+\eta,\frac{\pi}{2}-\eta\right]\qquad\text{and}\qquad \arg e(b\theta)G_2(t)\in \left[-\frac{\pi}{2},\frac{\pi}{2}\right]
\]
has the measure $\geq\tfrac{1}{8|b||a|}$ , so for sufficiently small $l$, depending only on $a$ and $b$, this set contains a rational number of the form $\tfrac{k^*}{l}$, and \eqref{eq:arg} holds by taking $k\equiv k^* \bmod l$ with $0\leq k<l$.
\end{proof}

\section{Proof of Theorem \ref{th:main}}

In order to prove Theorem \ref{th:main} we shall use some results from the theory of diophantine approximation. 

Let us recall that a vector $\mathbf{x}\in\mathbb{R}^n$ belongs to $\gamma\subset\mathbb{R}^n$ mod $1$ if there exists a vector $\mathbf{y}\in\mathbb{Z}^n$ such that $\mathbf{x}-\mathbf{y}\in\gamma$.
\begin{thm}[Kronecker]\label{thm:Kronecker}
Let $\alpha_1,\ldots,\alpha_n$ be real numbers linearly independent over $\mathbb{Q}$ and $\gamma$ be a subregion of the $n$-dimensional unit cube with Jordan measure $m(\gamma)$. Then
\[
\lim_{T\to\infty}\frac{1}{T}\operatorname{meas}\left\{\tau\in (0,T]: (\alpha_1\tau,\ldots,\alpha_n\tau)\in\gamma\bmod 1\right\} = m(\gamma).
\]
\end{thm}
\begin{proof}
This is \cite[Theorem A.8.1]{KV}.
\end{proof}

We say that the curve $\gamma(\tau): [0,\infty]\to\mathbb{R}^n$ is \emph{uniformly distributed mod $1$} in $\mathbb{R}^n$ if for every $\alpha_j,\beta_j$, $j=1,2,\ldots,n$, with $0\leq \alpha_j<\beta_j\leq 1$ we have
\[
\lim_{T\to\infty}\frac{1}{T}\operatorname{meas}\left\{\tau\in (0,T]: \gamma(\tau)\in[\alpha_1,\beta_1]\times\cdots\times[\alpha_n,\beta_n]\bmod 1\right\} = \prod_{j=1}^n (\beta_j-\alpha_j).
\]

\begin{lemma}\label{lem:Int}
Let $\gamma(\tau)$ be uniformly distributed mod $1$ in $\mathbb{R}^n$ and $X$ be a closed and Jordan measurable subregion of the unit cube in $\mathbb{R}^n$. Suppose that $\Omega$ is a family of complex-valued continuous functions defined on $X$. If $\Omega$ is uniformly bounded and equivcontinuous, then, uniformly on $\Omega$, we have
\[
\lim_{T\to\infty}\frac{1}{T}\int_{A_T}f(\{\gamma(\tau)\})d\tau = \idotsint_X f(x_1,\ldots,x_n) \prod_{j=1}^n dx_j,
\]
where $A_T$ denotes the set of $\tau\in (0,T]$ such that $\gamma(\tau)\in X$ mod $1$ and $\{\gamma(\tau)\}$ denotes the fractional part of $\gamma(\tau)$.
\end{lemma}
\begin{proof}
This is \cite[Theorem A.8.3]{KV}.
\end{proof}

Moreover, we shall need the following result due to Mergelyan.
\begin{thm}[Mergelyan]
Let $K$ be a compact set with connected complement and $f(s)$ be continuous on $K$ and analytic in the interior of $K$. Then, for every $\varepsilon>0$, there exists a polynomial $P(s)$ such that
\[
\max_{s\in K} |f(s)-P(s)|<\varepsilon.
\]
\end{thm}
\begin{proof}
See \cite[Chapter III]{G}.
\end{proof}

\begin{proof}[Proof of Theorem \ref{th:main}]
By Mergelyan's theorem it suffices to assume that $f_a,f_b$ are polynomials without zeros on $K$. Then we can find an admissible set $U$ such that $f_a$, $f_b$ have no zeros on the closure of $U$ and $K\subset U\subset\overline{U}\subset\{s\in\mathbb{C}:\tfrac{1}{2}<\Re(s)<1\}$.

Let us fix $\varepsilon>0$. By Lemma \ref{lem:dense}, for any $y>0$, there exists a finite set of primes $M$ containing all primes $p\leq y$ and the sequence $\theta_p$, $p\in M$, such that
\[
\max_{c\in\{a,b\}}\max_{s\in \overline{U}}|\zeta_M(s;c\theta) - f_c(s)|<\varepsilon.
\]
Moreover, if 
\begin{equation}\label{eq:LogP}
\max_{p\in M}\left\Vert \tau\frac{\log p}{2\pi} - \theta \right\Vert<\delta\qquad\text{for sufficiently small $\delta$},
\end{equation}
then 
\[
\max_{c\in\{a,b\}}\max_{p\in M}\left\Vert \tau\frac{c\log p}{2\pi} - c\theta\right\Vert<c\delta,
\]
and, by continuity,
\begin{equation}\label{eq:cont}
\max_{c\in\{a,b\}}\max_{s\in\overline{U}}|\zeta_M(s+ic\tau;\mathbf{0}) - f_c|<\varepsilon;
\end{equation}
here $\mathbf{0}$ is the sequence of zeros and $||\cdot||$ denotes the distance to the nearest integer.

Put $Q:=\{p:p\leq z\}$ for $z>y$ such that $M\subset Q$ and define the set
\[
\mathcal{D} :=\{(\omega_p)_{p\in Q}:\max_{p\in M}\Vert\omega_p-\theta_p\Vert<\delta\}.
\]
Next let us consider
\[
S = \sum_{c\in\{a,b\}}\frac{1}{T}\int_{A_T}\left(\iint_U \left|\zeta(s+ic\tau)-\zeta_M(s+ic\tau;\mathbf{0})\right|^2 d\sigma dt\right)d\tau,
\]
where $T>1$ and $A_T$ is the set of $\tau\in [0,T]$ satisfying \eqref{eq:LogP}.

By Cauchy-Schwarz inequality, we get $S \leq 2S_1+2S_2$, where
\[
S_1=\sum_{c\in\{a,b\}}\frac{1}{T}\int_{A_T}\left(\iint_U \left|\zeta_Q(s+ic\tau,\mathbf{0})-\zeta_M(s+ic\tau;\mathbf{0})\right|^2 d\sigma dt\right)d\tau
\]
and
\[
S_2=\sum_{c\in\{a,b\}}\frac{1}{T}\int_{A_T}\left(\iint_U \left|\zeta(s+ic\tau)-\zeta_Q(s+ic\tau;\mathbf{0})\right|^2 d\sigma dt\right)d\tau.
\]

By the unique factorization of integers, the curve $\gamma(\tau) = (\tau\tfrac{\log p}{2\pi})_{p\in Q}$ is uniformly distributed modulo $1$ on $\mathbb{R}^{\pi(z)}$. Hence, by Lemma \ref{lem:Int} and the fact that there is only restriction on $\omega_p$ with $p\in M$ in the definition of the set $\mathcal{D}$, we have
\begin{align*}
&\lim_{T\to\infty}\frac{1}{T}\int_{A_T}\left|\zeta_Q\left(s,\left(\tau\frac{c\log p}{2\pi}\right)\right)-\zeta_M\left(s;\left(\tau\frac{c\log p}{2\pi}\right)\right)\right|^2 d\tau\\
 &\qquad\qquad\qquad\qquad= \idotsint\displaylimits_{\mathcal{D}}\left|\zeta_M(s;(c\omega_p))\right|^2\left|\zeta_{Q\setminus M}(s;(c\omega_p))-1\right|^2 \prod_{p\in Q}d\omega_p\\
  &\qquad\qquad\qquad\qquad\leq \left(\max_{s\in \overline{U}}|f_c(s)|+\varepsilon\right)^2 \idotsint\displaylimits_{\mathcal{D}}\left|\zeta_{Q\setminus M}(s;(c\omega_p))-1\right|^2 \prod_{p\in Q}d\omega_p\\
   &\qquad\qquad\qquad\qquad\ll m(\mathcal{D})\int_0^1\ldots\int_0^1\left|\zeta_{Q\setminus M}(s;(c\omega_p))-1\right|^2 \prod_{p\in Q\setminus M}d\omega_p.
\end{align*}
Moreover, by easy calculation and the fact that $Q\setminus M$ contains only primes greater than $y$, one can show that
\[
\int_0^1\ldots\int_0^1\left|\zeta_{Q\setminus M}(s;(c\omega_p))-1\right|^2 \prod_{p\in Q\setminus M}d\omega_p\leq \sum_{n>y}\frac{1}{n^{2\sigma}},\qquad s=\sigma+it\in\overline{U}.
\]
Therefore, since $\sigma>\tfrac{1}{2}$ for all $s\in\overline{U}$, we have
\[
S_1\leq \frac{1}{4}m(\mathcal{D})\varepsilon^2
\]
for sufficiently large $y>0$.

Now, using the well-know estimate for the mean-square of the Riemann zeta function and Carlson's theorem (see \cite[Theorem A.2.10]{KV}) gives that
\[
S_2\leq \frac{1}{4}m(\mathcal{D})\varepsilon^2
\]
for sufficiently large $z>0$, and we have
\[
S\leq m(\mathcal{D})\varepsilon^2.
\]

On the other hand, we know that the sequence $\tfrac{\log p}{2\pi}$, $p\in Q$, is linearly independent over $\mathbb{Q}$, so by the Kronecker approximation theorem, we get
\[
\lim_{T\to\infty}\frac{1}{T}\int_{A_T}d\tau = m(\mathcal{D}).
\]
Thus, by \eqref{eq:cont}, one can show that the set of $\tau\in A_T$ satisfying
\[
\max_{c\in\{a,b\}}\iint_{U}|\zeta(s+ic\tau)-f_c(s)|^2d\sigma dt \ll \varepsilon^2
\]	
has measure $\gg T$. Therefore, since approximation in respect to $L^2(U)$-norm implies uniform approximation on $K\subset U$ (see eg. \cite[Chapter I, Section 1, Lemma 1]{G}), the proof is complete.
\end{proof}

\end{document}